\documentclass[12pt]{article}
\usepackage{amsmath}
\usepackage{amsfonts}
\usepackage{graphicx}
\usepackage{a4}
\newtheorem{theorem}{Theorem}

\newtheorem{lemma}[theorem]{Lemma}
\newtheorem{proposition}[theorem]{Proposition}

\newenvironment{proof}[1][Proof]{\noindent\textbf{#1.} }{\ \rule{0.5em}{0.5em}}
\newcommand\func[1]{\mathrm{#1}}
\begin{document}

\title{Unique normal forms for area preserving maps\\
near a fixed point with neutral multipliers}
\author{Vassili Gelfreich${}^1$\footnote{This work was partially
supported by a grant from the Royal Society.}
\and
Natalia Gelfreikh${}^2$
\\[24pt]
${}^1$\small Mathematics Institute, University of Warwick,\\
\small Coventry, CV4 7AL, UK\\
\small E-mail: v.gelfreich@warwick.ac.uk\\[12pt]
${}^2$\small Department of Higher Mathematics and Mathematical Physics,\\
\small Faculty of Physics, St.Petersburg State University, Russia\\
\small E-mail: gelfreikh@mail.ru}
\maketitle

\begin{abstract}
We study normal forms for families of area-preserving maps
which have a fixed point with neutral multipliers $\pm1$ at $\varepsilon=0$.
Our study covers both the orientation-preserving and orientation-reversing cases.
In these cases Birkhoff normal forms do not provide a substantial simplification
of the system. In the paper we prove that the Takens normal form vector field
can be substantially simplified. We also show that if certain 
non-degeneracy conditions are satisfied no
further simplification is generically possible
since the constructed normal forms are unique.
In particular, we provide a full system of formal
invariants with respect to formal coordinate changes.
\end{abstract}

%\tableofcontents
\newpage
\section{Introduction}
The theory of Birkhoff normal forms provides a powerful tool for studying
the local dynamics (see e.g.~\cite{AKN,Kuznetsov}).
The normal form theory uses changes of variables to
transform a map or a differential equation into a simpler
one called a normal form.
It is well known that the change of variables is usually not unique
and the freedom in its choice can be used to achieve
further simplifications. In this paper we consider
several problems where these additional simplifications
play the central role in the analysis of the normal form
and, under certain non-degeneracy conditions,
construct unique normal forms.

Similar results for flows on the plane were obtained by various authors
following the paper by Baider and Sanders \cite{BaiderS1991}.
In this paper we follow a rather classical approach to the problem.
In particular, we follow Takens \cite{Takens1974} in the approach to
normal forms for maps, the simplification procedure
is based on the notion of a linear grading function
introduced by Kokubu et al. \cite{KokubuOW1996}.
Finally, we use Lie series for constructing canonical changes of variables
which are preferable \cite{Deprit1969} as it generates
formulae which are easier for implementations.

We note that our procedure is based on a simplification of
the Birkhoff normal form. The characteristic property of
the Birkhoff normal is a symmetry generated by the linear
part of the original system at $\varepsilon=0$.
Therefore our results can be directly applied to systems
which posses the same symmetries independently of their origin.

In the paper all theorems are stated for families of maps,
but the assumptions of these theorems can be satisfied by an individual map $F_0$
which can be obviously considered as a family which is constant as a function of $\varepsilon$.
Therefore all theorems can be used to make conclusions about $F_0$ by
setting $\varepsilon=0$ in the corresponding equations.

\subsection{Orientation preserving families}
Let $F_{\varepsilon }:$ $\mathbb{R}^{2}\mathbb{\rightarrow R}^{2}$ be an
analytic family of  area and orientation preserving maps which depend on
a small parameter $\varepsilon$.
Assume that at $\varepsilon =0$ the origin is a fixed point:
\begin{equation*}
F_{0}(0)=0.
\end{equation*}%
Since $F_{0}$ is area and orientation preserving $\det DF_{0}(0)= 1$. Hence
we can denote the two eigenvalues of the Jacobian matrix
$DF_{0}(0)$ by $\mu_0 $ and $\mu_0 ^{-1}$.
These eigenvalues are often called {\em multipliers\/}
of the fixed point. The Birkhoff normal form depends
on the values of the multipliers.

If $\mu_0$ is real and $|\mu_0|\ne1$, the fixed point is hyperbolic
and Moser proved that the Birkhoff normal form is unique and analytic \cite{Moser1956}.
If $\mu_0$ is not real, the fixed point is elliptic (obviously $|\mu_0|=1$ in this case).
If the elliptic fixed point is not resonant, the Birkhoff normal form is unique.
In the resonant case, the unique normal form was recently obtained by the authors~\cite{GG2008}.

In this paper we study the case of a parabolic fixed point with $\mu_0=\mu_0^{-1}=\pm1$.
Though the word ``parabolic" is often used
for a generic fixed point with $\mu_0=1$ only.
If the Jacobian matrix is not diagonalisable there is an area-preserving linear change of coordinates
({\em i.e.}, defined by a matrix with the unit determinant),
after which we have
\begin{equation*}
DF_{0}(0)=\left(
\begin{array}{cc}
 \mu_0 & a \\
0 & \mu_0
\end{array}
\right)
\end{equation*}
with $a=\pm1$. A reflection in the vertical axes changes the sign of $a$.
Therefore without loosing in generality we assume
\begin{equation*}
DF_{0}(0)=\mu_0 \left(
\begin{array}{cc}
1  & 1 \\
0 & 1 %
\end{array}
\right).
\end{equation*}
If $\mu_0=1$ the classical theory of Birkhoff normal forms
does not provide any simplification for the map since all terms of the Taylor expansion
are ``resonant". On the other hand, it is well known that $F_\varepsilon$
can be formally interpolated by an autonomous Hamiltonian flow, i.e., there is a formal
Hamiltonian such that its time-one map coincides with the Taylor series of the map:
\begin{equation*}
F_{\varepsilon }=\Phi^1_{h_{\varepsilon }}\,,
\end{equation*}
where $\Phi^1_{h_{\varepsilon }}$ is the time-one map generated
by the Hamiltonian $h_\varepsilon$.
In particular $h_\varepsilon$ is a formal integral of the map: 
$h_\varepsilon\circ F_\varepsilon=h_\varepsilon$.
Generically the series $h_\varepsilon$ diverge \cite{GS2001}.
The formal interpolating Hamiltonian can be substantially simplified.

\begin{theorem} \label{th1}
Let $F_{\varepsilon }$ be an analytic (or $C^\infty$, or formal) family of area
preserving maps defined in a neighbourhood of the origin such that $F_0(0)=0$,
$DF_0(0)$ has a double eigenvalue $\mu_0=1$ and is not diagonalisable.
Then there exists a canonical%
\footnote{Possibly reflecting the orientation} formal change of variables such that \begin{equation*}
F_{\varepsilon }=
\phi_{\varepsilon }^{-1}\circ \Phi _{h_{\varepsilon }}^{1}\circ \phi_{\varepsilon }
\end{equation*}
with
\begin{equation}\label{Eq:hepssimple}
h_{\varepsilon }(x,y)=\frac{y^{2}}{2}+U_\varepsilon(x)\,,
\end{equation}
where
\begin{equation*}
U_{\varepsilon }(x)=\sum_{k+3m \ge 3} u_{km}x^k \varepsilon^m
\end{equation*}
is a formal series in two variables.
\end{theorem}

This form of the interpolating Hamiltonian is very convenient
for studying bifurcations. Indeed, a truncation
of the series $h_\varepsilon$ provides a good approximation
for the map. On the other hand, the phase portrait of a Hamiltonian
written in the form (\ref{Eq:hepssimple}) is easy to analyse.

Under an additional non-degeneracy assumption, a further simplification of $h_\varepsilon$
is possible and leads to the unique normal form.
The coefficients of the unique normal
form are invariants of the map under formal (and consequently
analytic) changes of variables. It is interesting to note
that the unique normal form still contains infinitely
many coefficients. Consequently the number of independent
formal invariants is infinite.

\begin{theorem} \label{th3}
Let $F_{\varepsilon }$ satisfy the assumptions of Theorem~\ref{th1}.
If there is $n\ge3$ such that $u_{k0}=0$ for all $k<n$ and $b:=u_{n0}\ne0$,
then there exists a canonical formal change of variables such that
$h_\varepsilon$ takes the form\/ {\rm (\ref{Eq:hepssimple})\/} with
\begin{equation*}
U_{\varepsilon }(x)=bx^n+ \sum_{\substack {k+nm > n\\ k\ne-1\mod n}} u_{km}x^k \varepsilon^m.
\end{equation*}
The coefficients of the series are defined uniquely by the map $F_{\varepsilon}$.
\end{theorem}

We note that for $\varepsilon=0$ the theorem follows from the previous one
in combination with the result by Baider and Sanders \cite{BaiderS1991}
who proved that a planar Hamiltonian
with a nilpotent singularity has the following unique normal form
\[
h_0=\frac{y^2}{2}+\sum_{\substack {k\ge n  \\ k\ne-1\mod n}}b_kx^k\,,
\]
where $n$ corresponds to the lowest non-vanishing order, i.e., $b_{n}\ne0$.

In the case of $\mu_0=-1$ the Birkhoff normal form
is odd in $(x,y)$ and the Takens normal form vector field is
described by an even Hamiltonian.

\begin{theorem} \label{th4}
Let $F_{\varepsilon }$ be an analytic (or $C^\infty$, or formal) family of  area-preserving maps
defined in a neighbourhood of the origin such that $F_0(0)=0$,
$DF_0(0)$ has a double eigenvalue $\mu_0=-1$ and is not diagonalisable.
Then there exists a canonical formal change of variables such that
\begin{equation*}
F_{\varepsilon }=- \phi_\varepsilon ^{-1}\circ \Phi _{h_{\varepsilon }}^{1}\circ \phi_\varepsilon ,
\end{equation*}
where
\begin{equation*}
h_{\varepsilon }(x,y)=\frac{y^{2}}{2}+U_\varepsilon (x)\qquad
\mbox{with}\qquad U_\varepsilon(x)= \sum_{k+2m \ge 2} u_{km}x^{2k} \varepsilon^m .
\end{equation*}
Moreover, the coefficients
of the series are defined uniquely by the family $F_{\varepsilon}$.
\end{theorem}

It is interesting to note that in this case the uniqueness of the coefficients
does not require any additional non-degeneracy condition.

Finally we consider a family with diagonalisable $DF_0(0)$.

\begin{theorem} \label{th5}
Let $F_{\varepsilon }$ be an analytic (or $C^\infty$, or formal) family of area
preserving maps defined in a neighbourhood of the origin such that $F_0(0)=0$,
$DF_0(0)$ has a double eigenvalue $\mu_0=1$ or $\mu_0=-1$ and is diagonalisable
and a non-degeneracy condition is satisfied.
Then there exists a canonical formal change of variables $\phi_{\varepsilon }$ such that \begin{equation*}
F_{\varepsilon }=
\mu_0\,\phi_{\varepsilon }^{-1}\circ \Phi _{h_{\varepsilon }}^{1}\circ \phi_{\varepsilon }
\end{equation*}
with
\begin{equation}\label{Eq:hepssimplediag}
h_{\varepsilon }(x,y)=xy^{2}+ax^3+A(x,\varepsilon)+y B(xy^2,\varepsilon)\,,
\end{equation}
where $A$ and $B$ are formal series of the form
\begin{equation}
A(x,\varepsilon)=\sum_{k+3m \ge 4} a_{km}x^k \varepsilon^m\,,
\qquad
B(xy^2,\varepsilon)=\sum_{k+m\ge 1}b_{km}{(xy^2)}^k \varepsilon^m \,.
\end{equation}
If $\mu_0=-1$ the formal Hamiltonian $h_\varepsilon$ is even.
\end{theorem}

Note that this theorem does not include the uniqueness statement. 
We will prove that the coefficients of the Hamiltonian $h_{\varepsilon }$
are invariant under formal tangent-to-identity canonical changes
which preserve the form of the series. Unlike the previous theorems
the uniqueness of the normal form does not follow  in the case of
a negative $a$ because there are two linear changes of variables which do not change the
form of the leading part of the series but may affect higher orders.
As a result there are possibly three different $h_\varepsilon$
corresponding to a single family~$F_{\varepsilon}$.

\subsection{Orientation reversing families}
Now let us assume that $F_{\varepsilon }:$ $\mathbb{R}^{2}\mathbb{\rightarrow R}^{2}$
preserves the area but reverses the orientation.
Similar to the previous subsection we assume that at $\varepsilon =0$
the origin is a fixed point:
\begin{equation*}
F_{0}(0)=0.
\end{equation*}%
Since $\det DF_{0}(0)=- 1$, we can denote
the multipliers of the origin by $\mu_0 $ and~$-\mu_0 ^{-1}$.
Since $F_0$ is real, both multiplies are necessarily real.
We expect that in the hyperbolic case a result similar to
\cite{Moser1956} should be valid. In this paper we consider
the parabolic case which corresponds to $\mu_0=-1$. Then there is a canonical change of
coordinates which diagonalises the Jacobian and we can assume
without loosing in generality:
$$
DF_0(0)=\left(\begin{array}{rr}-1&0\\0&1\end{array}\right)\,.
$$
According to the classical normal form theory,
the Hamiltonian of the Takens normal form vector field
is an odd function in $x$, i.e., $h_\varepsilon(-x,y)=-h_\varepsilon(x,y)$.
The formal series $h_0$ starts with cubic terms, therefore
in the leading order
$$
h_0(x,y)=bxy^2 +ax^3\,.
$$
This Hamiltonian has the same form as in the orientation-preserving
case with a diagonalisable Jacobian. This similarity leads to the
similarity in the structure of the normal form.

\begin{theorem} \label{th5a}
Let $F_{0}$ be an analytic (or $C^\infty$, or formal)  area-preserving
orientation-reversing map defined in a neighbourhood of the origin such that $F_0(0)=0$,
$DF_0(0)$ has eigenvalues $-1$ and $1$,
and the leading order of the normal form is non-degenerate, i.e, $b\ne 0$.
Then there exists a canonical formal change of variables such that
\begin{equation*}
F_{\varepsilon }={\mathrm{Diag}}\,(-1,1) \,\phi_0^{-1}\circ \Phi _{h_{0 }}^{1}\circ \phi_0 ,
\end{equation*}
where
\begin{equation*}
h_{0}(x,y)=xy^2 +ax^3+A(x,\varepsilon)+y B(xy^2,\varepsilon)\,,
\end{equation*}
where $A$ and $B$ are formal series of the form
\begin{equation*}
A(x,\varepsilon)=\sum_{k+3m \ge 4} a_{mk} x^{2k-1}\varepsilon^m\,,
\qquad
B(xy^2,\varepsilon)= \sum_{k+m \ge 1} b_{mk}  {(xy^2)}^{2k-1}\varepsilon^m \, .
\end{equation*}
The coefficients of the series are defined uniquely by the map.
\end{theorem}

\subsection*{Structure of the paper}
In Section~\ref{Se:formser} we
describe the usage of formal series and
quasi-homogeneous polynomials.
In Section~\ref{Formal-int} we prove a theorem
about formal interpolation of a map by an autonomous Hamiltonian.
In Section~\ref{Se:mu1} we derive the unique normal 
form for the case $\mu_0=1$ and non-diagonalisable $DF_0(0)$.
In Section~\ref{Se:Birkhoff} the Birkhoff normal form is derived.
In Section~\ref{Se:mu-1} we study the case $\mu_0=-1$.
The case of diagonalisable Jacobian is considered in Section~\ref{Se:diag}.
Finally, in Section~\ref{orientrev} we consider orientation reversing families.

The results of Sections~\ref{Formal-int} and \ref{Se:Birkhoff} are used in the proof of the main theorems
and are included for completeness of the arguments.

\section{Formal series and quasi-homogeneous polynomials}\label{Se:formser}

In this paper we will be mainly interested in transformations given in the form
of formal power series. We will consider the series in powers of
the space variables $(x,y)$ and the parameter $\varepsilon$. The series
have the form
\[
g(x,y,\varepsilon)=\sum_{k,l,m}c_{klm}x^ky^l\varepsilon^m\,.
\]
A formal series is treated as a collection of coefficients.
Formal series form an infinite dimensional vector space.
The addition, multiplication, integration
and differentiation are defined in a way compatible with the common
definition on the subset of convergent series.

As the series involve several variables, it is convenient
to group terms which are ``of the same order". A usual
choice is to consider $x$, $y$ and $\varepsilon $ to be of the same order.
But for purpose of this paper it is much more convenient
to assume
\begin{itemize}
\item $x$ is of order $k_0$,
\item $y$ is of order $l_0$,
\item $\varepsilon$ is of order $m_0$,
\end{itemize}
where $k_0,l_0,m_0$ will be chosen later.
Then a monomial $x^ky^l\varepsilon^m$ is considered to be
of order $kk_0+ll_0+mm_0$ (this is a grading function in terminology of \cite{KokubuOW1996}).
We can write
\[
g(x,y,\varepsilon)=\sum_{p\ge 0} g_p(x,y,\varepsilon)
\]
where
\[
g_p(x,y,\varepsilon)=\sum_{kk_0+ll_0+mm_0=p}c_{klm}x^ky^l\varepsilon^m
\]
is a quasi-homogeneous polynomial of order $p$
since it has the property
\[
g_p(\lambda^{k_0} x,\lambda^{l_0} y,\lambda^{m_0} \varepsilon)=\lambda^p g_p(x,y,\varepsilon)
\]
for any $\lambda\in\mathbb{R}$. We stress that this notation does not refer
to a resummation of the divergent series but simply indicates the order
in which  the coefficients of the series are to be treated.

In order to give a rigorous background for manipulation with
formal series we remind that the space of formal series ${\mathfrak H}$
can be considered as an infinite dimensional vector space
equipped with the following metric.
Let $g$ and $\tilde g$ be two formal
series and let $p$ denote the lowest (quasi-homogeneous)
order of $g-\tilde g$. If $g\ne\tilde g$ then $p$ is
finite and
$$
d(g,\tilde g)=2^{-p},
$$
otherwise
$$
 d(g,g)=0\,.
$$
It can be check that $(\mathfrak H,d)$ is a complete metric space.
Moreover polynomials are dense in $(\mathfrak H,d)$.
Hence we can define formal convergence and formal continuity
on the space of formal series. In particular an operator
is formally continuous if each coefficients of a series
in its image is a function of a finite number of
coefficients of a series in its argument.

We note that any of the series
involved in next definitions may diverge.

Let $\chi$ and $g$ be two formal power series.
The linear operator defined by the formula
\begin{equation}\label{Def_L}
L_{\chi }g=\left\{g, \chi\right\}
\end{equation}
is called {\em the Lie derivative\/} generated by $\chi$.
We note that if $\chi$ starts with an order $p$ and $g$
starts with an order $q$, then the series $L_{\chi }g$
starts with the order $p+q-(k_0+l_0)$ as the Poisson bracket
involves differentiation with respect to $x$ and $y$.

If
\begin{equation}\label{Eq:p-large}
p\ge k_0+l_0+1
\end{equation}
the lowest order in $L_{\chi }g$ is at least $q+1$.
Then we define the exponent of $L_\chi$ by
\begin{equation}\label{Eq:formalexp}
\exp(L_\chi)g=\sum_{k\ge0}\frac1{k!}L_\chi^kg\,,
\end{equation}
where $L_\chi^k$ stands for the operator $L_\chi$
applied $k$ times. The lowest order in the series $L_\chi^kg$
is at least $q+k$. So the series
(\ref{Eq:formalexp}) converges with respect to the metric $d$,
i.e., each coefficient of the result depends on a finite number
of coefficients of the series $\chi$ and $g$.

\medskip

We consider the formal series
\[
\Phi_\chi^1 (x,y)=\bigl(\exp(L_\chi) x,\,
\exp(L_\chi) y
\bigr).
\]
We say that $\Phi^1_\chi$ is a Lie series generated by the
formal Hamiltonian $\chi$. If $\chi$ is polynomial
the series converge on a poly-disk and coincide with
a map which shifts points along trajectories of the
Hamiltonian system with Hamiltonian function $\chi$.
For this reason we will call $\Phi^1_\chi$
a time-one map of the Hamiltonian $\chi$
even in the case when the series do not converge.

We note that it is easy to construct the formal series
for the inverse map:
\[\Phi_\chi^{-1}(x,y)=\bigl(\exp(-L_\chi) x,\,
\exp(-L_\chi)  y
\bigr)\,.
\]
Then $\Phi_{\chi}^1\circ\Phi_{\chi}^{-1}(x,y)=(x,y)$.
We also note that
\[
g\circ \Phi_\chi^1=\exp(L_\chi)g\,.
\]
These formula are well known for convergent
series and can be extended onto ${\mathfrak H}$ due to the
density property.

\section{Formal interpolation}\label{Formal-int}

In this section it will be convenient to set $k_0=2$, $l_0=3$ and $m_0=6$.
So we will order terms in power series supposing that
\begin{equation*}
\begin{array}{cc}
x & \text{is of order }2, \\
y & \text{is of order }3, \\
\varepsilon  & \text{is of order }6.%
\end{array}%
\end{equation*}

Next theorem states that $F_\varepsilon$ can be formally
interpolated by an autonomous Hamiltonian flow.

\begin{theorem} \label{interpol}
Let $F_{\varepsilon }$ be a family of area-preserving maps such that $F_0(0)=0$ and
\begin{equation*}
DF_{0}(0)=\left(
\begin{array}{cc}
1 & 1 \\
0 & 1%
\end{array}%
\right)
\end{equation*}%
then there exists a unique (up to adding a formal series
in $\varepsilon$ only) formal Hamiltonian
\begin{equation*}
h_{\varepsilon }(x,y)=
\sum_{p\geq 6}h_{p}(x,y,\varepsilon),
\qquad\mbox{where}\quad
h_{p}(x,y,\varepsilon)=\sum_{2k+3l+6m=p}h_{klm}x^{k}y^{l}\varepsilon ^{m}\,,
\end{equation*}
such that
\begin{equation}\label{Eq:forint}
\Phi _{h_{\varepsilon }}^{1}=F_{\varepsilon }.
\end{equation}
Moreover, if $F_\varepsilon$ is odd in $(x,y)$ then $h_\varepsilon$ is even.
\end{theorem}

Before proving the theorem we need a simple technical statement.
\begin{lemma}\label{Lemmadiv}
Let $f_p$ and $g_{p+1}$ be quasi-homogeneous polynomial of orders $p$ and $p+1$ respectively.
There is a quasi-homogeneous polynomial
$h_{p+3}$ of order $p+3$ which satisfies the system
\begin{equation}\label{Eq:div1} \begin{array}{c}
\partial_y h_{p+3}=f_p \\
-\partial_x h_{p+3}=g_{p+1}
\end{array}
\end{equation}
if and only if
\begin{equation}\label{Eq:div2}
\func{div} \left(
\begin{array}{c}
f_p \\
g_{p+1}
\end{array}%
\right) =0\,.
\end{equation}
If exists, the polynomial $h_{p+3}$ is unique
in the class of quasi-homogeneous polynomials\/
{\rm (}up to adding $h_{00m}\varepsilon^m${\rm )}.
Moreover, if $f_p$ and $g_{p+1}$ are odd in $(x,y)$
then $h_{p+3}$ is even.
\end{lemma}

\begin{proof}
The polynomials can be written in the form:
\begin{eqnarray}
\nonumber
f_p &=& \sum_{2k+3l+6m=p} a_{klm} x^ky^l \varepsilon^m \,,\\
\label{Eq:fpgp+1}
g_{p+1} &=&\sum_{2k+3l+6m=p+1} b_{klm} x^ky^l \varepsilon^m\,,\\
\nonumber
h_{p+3}&=& \sum_{2k+3l+6m=p+3} h_{klm} x^ky^l \varepsilon^m\,.
\end{eqnarray}
Substituting these polynomials into the equations and collecting
similar terms we see that (\ref{Eq:div2}) is equivalent to
\begin{equation}\label{Eq:t1}
ka_{k,l-1,m}+lb_{k-1,l,m}=0, \quad k,l\ge0
\end{equation}
and (\ref{Eq:div1}) is equivalent to
\begin{equation}\label{Eq:t2}
h_{klm}= \frac{a_{k,l-1,m}}{l}\quad \mbox{for $l>0$} \quad \mbox{and} \quad
h_{klm}=-\frac{b_{k-1,l,m}}{k}\quad \mbox{for $k>0$}\,.
\end{equation}
Equalities (\ref{Eq:t2}) are compatible if and only if (\ref{Eq:t1}) is satisfied.
We note that the equations do not involve the coefficient $h_{00m}$ which can be chosen arbitrarily.

If $f_p$ and $g_{p+1}$ are odd in $(x,y)$, i.e. $a_{klm}=0$ and $b_{klm}=0$ for $k+l$ even,
then $h_{klm}=0$ for $k+(l-1)$ even (or equivalently for $k+l$ odd), i.e. $h$ is even.
\end{proof}

\medskip

\begin{proof}[Proof of Theorem~\ref{interpol}]
The Taylor series for $F_\varepsilon: (x,y)\mapsto (x_{1},y_{1})$
can be written as a sum of quasi-homogeneous polynomials:
\begin{equation}
\left\{
\begin{array}{l}
x_{1}=x+ y+\sum\limits_{p\geq 4}f_{p}(x,y,\varepsilon ), \\
y_{1}=y+ cx^2+ \sum\limits_{p\geq 4}g_{p+1}(x,y,\varepsilon ) ,
\end{array}
\right.   \label{Lesson7_1}
\end{equation}
where $f_{p}$ and $g_{p+1}$ are quasi-homogeneous polynomials:
\begin{equation*}
\left\{
\begin{array}{c}
f_{p}(x,y,\varepsilon )=\sum\limits_{2k+3l+6m=p}f_{klm}x^{k}y^{l}\varepsilon
^{m}, \\[12pt]
g_{p+1}(x,y,\varepsilon )=\sum\limits_{2k+3l+6m=p+1}g_{klm}x^{k}y^{l}\varepsilon
^{m}.
\end{array}
\right.
\end{equation*}
The time-one map of $h_{\varepsilon }$ is given by
the Lie series
\begin{equation*}
\Phi _{h_{\varepsilon }}^{1}(x,y)=\left(
\begin{array}{c}
x+L_{h_{\varepsilon }}x+\sum\limits_{k\geq 2}\frac{1}{k!}L_{h_{\varepsilon
}}^{k}x \\
y+L_{h_{\varepsilon }}y+\sum\limits_{k\geq 2}\frac{1}{k!}L_{h_{\varepsilon
}}^{k}y%
\end{array}%
\right) ,
\end{equation*}
where the operator $L_{h_{\varepsilon }}$ is defined by
\begin{equation*}
L_{h_{\varepsilon }}(\varphi )=\left\{ \varphi ,h_{\varepsilon }\right\} =%
\frac{\partial \varphi }{\partial x}\frac{\partial h_{\varepsilon }}{%
\partial y}-\frac{\partial \varphi }{\partial y}\frac{\partial
h_{\varepsilon }}{\partial x}.
\end{equation*}
We note that the inequality (\ref{Eq:p-large}) is satisfied and consequently
the formal series is well defined.

Now we use induction to show that there is a formal Hamiltonian
$h_\varepsilon$ such that the formal series $\Phi _{h_{\varepsilon }}^{1}$
coincides with the Taylor expansion of $F_{\varepsilon }$ at all orders.
We note the first component of the series
starts with the quasi-homogeneous order 2 (i.e. with $x$)
and the second one starts with 3 (i.e. with $y$).
Therefore it is convenient to consider an order $p$ in the first
component simultaneously with the order $p+1$ in the second one.
In this situation we say that we consider a term of the order $(p,p+1)$.
The terms of order $(2,3)$ in (\ref{Eq:forint}) coincide for an arbitrary $h_\varepsilon$.

To make the combinatorics easier
we introduce the notation
\begin{equation*}
L_{s}\varphi =\left\{ \varphi ,h_{s+5}\right\} .
\end{equation*}
Then
\begin{equation*}
L_{h_{\varepsilon }}\varphi =\left\{ \varphi ,h_{\varepsilon }\right\}
=\sum_{p\geq 6}\left\{ \varphi ,h_{p }\right\} =\sum_{s\geq
1}L_{s}\varphi \,.
\end{equation*}
Note that $L_{s}$ maps any quasi-homogeneous polynomial of order $j$ into a
quasi-homogeneous polynomial of order $j+s$. So we have
\begin{eqnarray*}
L_{h_{\varepsilon }}^{k}\varphi _{j} &=&
\left( \sum_{s\geq 1}L_{s}\right)^{k}
\varphi_j
=\sum_{\substack{ s_{1}+s_{2}+\ldots +s_{k}\geq k \\ %
s_{1},s_{2},\ldots ,s_{k}\geq 1}}L_{s_{1}}\ldots L_{s_{k}}\varphi _{j} \\
&=&
\sum_{m\geq k}\sum_{\substack{ s_{1}+s_{2}+\ldots +s_{k}=m \\ %
s_{1},s_{2},\ldots ,s_{k}\geq 1}}
L_{s_{1}}\ldots
L_{s_{k}}\varphi _{j}.
\end{eqnarray*}%
Now let us consider the components of $\Phi _{h_{\varepsilon }}^{1}$ at
order $(p,p+1)$. The first component is%
\begin{equation*}
\left[ L_{h_{\varepsilon }}x+\sum_{k\geq 2}\frac{1}{k!}L_{h_{\varepsilon
}}^{k}x\right] _{p}=L_{p-2}x+\sum_{k=2}^{p-2}\frac{1}{k!}\sum_{s_{1}+s_{2}+%
\ldots s_{k}=p-2}L_{s_{1}}\ldots L_{s_{k}}x
\end{equation*}%
and the second component%
\begin{equation*}
\left[ L_{h_{\varepsilon }}y+\sum_{k\geq 2}\frac{1}{k!}L_{h_{\varepsilon
}}^{k}y\right] _{p+1}=L_{p-2}y+\sum_{k=2}^{p-2}\frac{1}{k!}%
\sum_{s_{1}+s_{2}+\ldots s_{k}=p-2}L_{s_{1}}\ldots L_{s_{k}}y.
\end{equation*}%
Since
\begin{eqnarray*}
L_{p-2}x &=&\left\{ x,h_{p+3}\right\} , \\
L_{p-2}y &=&\left\{ y,h_{p+3}\right\}
\end{eqnarray*}%
the equation (\ref{Eq:forint}) can be rewritten as
\begin{equation}
\label{Eq:hp3}
\left(
\begin{array}{c}
\frac{\partial h_{p+3}}{\partial y} \\
-\frac{\partial h_{p+3}}{\partial x}%
\end{array}%
\right) =\left(
\begin{array}{c}
f_{p}-\sum\limits_{k=2}^{p-2}\frac{1}{k!}\sum\limits_{s_{1}+\ldots
+s_{k}=p-2}L_{s_{1}}\ldots L_{s_{k}}x \\
g_{p+1}-\sum\limits_{k=2}^{p-2}\frac{1}{k!}\sum\limits_{s_{1}+\ldots
+s_{k}=p-2}L_{s_{1}}\ldots L_{s_{k}}y%
\end{array}%
\right) .
\end{equation}

First we check that the equations can be solved for $p=3$.
Since
\begin{eqnarray*}
f_3 &=&y,\\
g_4 &=&cx^2
\end{eqnarray*}
the equations (\ref{Eq:hp3})
take the form
\begin{eqnarray}
 \frac{\partial h_{6}}{\partial y}&=&y,  \label{Lesson7_4} \\
-\frac{\partial h_{6}}{\partial x} &=&cx^2.  \notag
\end{eqnarray}
These equations can be easily solved:
\begin{equation}\label{h6c}
h_{6}=\frac{y^{2}}{2}-c\frac{x^3}{3}.
\end{equation}

Let us now proceed with the induction step. Suppose that for some $p\ge4$ we have
$h_{6},\ldots,h_{p+2}$ such that ${\left[ \Phi^1_{h_\varepsilon} \right]}_{(s,s+1)}
={\left[ F_{\varepsilon} \right]}_{(s,s+1)}$ for every $s \le p-1$ where 
$h_\varepsilon =\sum_{k \ge 6} h_k$ with arbitrary $h_k$ for $k>p+2$.
We want to find $h_{p+3}$ such that ${\left[ \Phi^1_{h_\varepsilon} \right]}_{(s,s+1)}
={\left[ F_{\varepsilon} \right]}_{(s,s+1)}$ for every $s \le p$ where 
$h_\varepsilon =\sum_{k \ge 6} h_k$ with arbitrary $h_k$ for $k>p+3$.

The quasi-homogeneous polynomial $h_{p+3}$ is defined from equations (\ref{Eq:hp3})
uniquely up to a function of $\varepsilon$
if and only if the right hand side is divergence free,
which follows from the area-preservation property
due to the following lemma.

\begin{lemma} \label{areapres}
If two families of area-preserving maps for which
$F_{\varepsilon}(0)=\tilde F_{\varepsilon}(0)=0$ and
\begin{equation*}
DF_{\varepsilon}(0)=D\tilde F_{\varepsilon}(0)= \left(
\begin{array}{cc}
1  & 1 \\
0 & 1 %
\end{array}
\right)
\end{equation*}
coincide up to the order $\left( p-1,p\right) $
then the divergence of the respective maps at the order $(p,p+1)$ are equal.
\end{lemma}

\begin{proof}
We can write the map in the form
\begin{equation*}
F_{\varepsilon }:\left(
\begin{array}{c}
x \\
y%
\end{array}%
\right) \mapsto \left(
\begin{array}{c}
x+ y+f(x,y,\varepsilon ) \\
 y+g(x,y,\varepsilon )%
\end{array}%
\right) .
\end{equation*}%
Since $F_{\varepsilon }$ is area-preserving we have%
\begin{equation*}
\det DF_{\varepsilon } =\left\vert
\begin{array}{cc}
1 +\partial _{x}f & 1+\partial _{y}f \\
\partial _{x}g & 1 +\partial _{y}g%
\end{array}%
\right\vert
=1
\end{equation*}
which is equivalent to
\begin{equation}
 \partial _{x}f+\partial _{y}g +\left\{ f,g\right\}
- \partial _{x}g\equiv0.
 \label{Parabolic_ap1}
\end{equation}
Similar identities are valid for $\tilde F_\varepsilon$.
We use the tilde to distinguish between the maps. We have $f_k=\tilde f_k$
and $g_{k+1}=\tilde g_{k+1}$ for $k\le p-1$.
We use the area-preserving property (\ref{Parabolic_ap1})
and collect the terms of order $p-2$
\begin{equation*}
\partial _{x}f_{p}+\partial _{y}g_{p+1}-\partial _{x}g_{p}
+\sum_{\substack{ k+l=p+3 \\ k,l\geq 4}}\left\{ f_{k},g_{l}\right\} =0,
\end{equation*}%
which we reorder as%
\begin{eqnarray*}
 \partial _{x}f_{p}+\partial _{y}g_{p+1}  &=&\partial
_{x}g_{p}-\sum_{\substack{ k+l=p+3 \\ k,l\geq 4}}\left\{ f_{k},g_{l}\right\}
\\
&=&\partial _{x}\widetilde{f}_{p}+\partial _{y}\widetilde{g}_{p+1}.
\end{eqnarray*}
So the divergences are equal.
\end{proof}

We finish the proof by the following observation. The right hand side of (\ref{Eq:hp3}) depends on $h_k$ with $k\le p+2$.
Let $\tilde h_\varepsilon=\sum_{k=6}^{p+2} h_k$. The maps $\Phi^1_{\tilde h_\varepsilon}$
and $F_\varepsilon$ satisfy the assumptions of the previous lemma
therefore their orders $(p,p+1)$ have the same divergence.
That is the solvability condition for the equation (\ref{Eq:hp3}).
The even property of $h_\varepsilon$ for odd $F_\varepsilon$ follows from Lemma \ref{Lemmadiv}.
\end{proof}

Now we will proof the interpolation theorem for a family with diagonalisable Jacobian.

\begin{theorem} \label{interpoldiag}
Let $F_{\varepsilon }=\left(
\begin{array}{c}
f_\varepsilon (x,y) \\
g_\varepsilon (x,y)%
\end{array}%
\right)$ be a family of area-preserving maps such that $F_0(0)=0$ and
\begin{equation*}
DF_{0}(0)=\left(
\begin{array}{cc}
1 & 0 \\
0 & 1%
\end{array}%
\right)
\end{equation*}%
then there exists a unique (up to adding a formal series
in powers $\varepsilon$ only) formal Hamiltonian
\begin{equation*}
h_{\varepsilon }(x,y)=
\sum_{p\geq 3}h_{p}(x,y,\varepsilon),
\qquad\mbox{where}\quad
h_{p}(x,y,\varepsilon)=\sum_{k+l+3m=p}h_{klm}x^{k}y^{l}\varepsilon ^{m}\,,
\end{equation*}%
such that
\begin{equation}\label{Eq:forint1}
\Phi _{h_{\varepsilon }}^{1}=F_{\varepsilon }.
\end{equation}
Moreover, if $f_\varepsilon$ is odd in $x$ and 
$g_\varepsilon$ is even in $x$ then $h_\varepsilon$ is odd in $x$.
\end{theorem}

\begin{proof}
We assume that
\begin{equation*}
\begin{array}{cc}
x & \text{is of order }1, \\
y & \text{is of order }1, \\
\varepsilon  & \text{is of order }3.%
\end{array}%
\end{equation*}
The rest of the proof is similar to Theorem~\ref{interpol}
so we briefly describe the necessary modifications.

We can write the Taylor series for $F_\varepsilon: (x,y)\mapsto (x_{1},y_{1})$
as a sum of quasi-homogeneous polynomials:
\begin{equation*}
\left\{
\begin{array}{l}
x_{1}=x+ \sum\limits_{p\geq 2}f_{p}(x,y,\varepsilon ), \\
y_{1}=y+ \sum\limits_{p\geq 2}g_{p}(x,y,\varepsilon ) ,
\end{array}
\right.
\end{equation*}
where $f_{p}$ and $g_{p}$ are quasi-homogeneous polynomials:
\begin{equation*}
\left\{
\begin{array}{c}
f_{p}(x,y,\varepsilon )=\sum\limits_{k+l+3m=p}f_{klm}x^{k}y^{l}\varepsilon
^{m}, \\[12pt]
g_{p}(x,y,\varepsilon )=\sum\limits_{k+l+3m=p}g_{klm}x^{k}y^{l}\varepsilon
^{m}.
\end{array}
\right.
\end{equation*}
In Lemma~\ref{Lemmadiv} instead of (\ref{Eq:div1}) we have
\begin{equation*} \begin{array}{c}
\partial_y h_{p+1}=f_p \\
-\partial_x h_{p+1}=g_{p}
\end{array}
\end{equation*}
instead of (\ref{Eq:div2}) we have
\begin{equation*}
\func{div} \left(
\begin{array}{c}
f_p \\
g_{p}
\end{array}%
\right) =0\,.
\end{equation*}
If $f_p$ is odd in $x$ and $g_p$ is even in $x$ then $h_{p+1}$ is odd in $x$.

Since
\begin{equation*}
F_{\varepsilon }:\left(
\begin{array}{c}
x \\
y%
\end{array}%
\right) \mapsto \left(
\begin{array}{c}
x+ f(x,y,\varepsilon ) \\
 y+g(x,y,\varepsilon )%
\end{array}%
\right) .
\end{equation*}%
instead of (\ref{Parabolic_ap1}) of Lemma~\ref{areapres}
 we have
\begin{equation*}
\func{div} \left(
\begin{array}{c}
f \\
g%
\end{array}%
\right) + \{ f,g \} =0,
\end{equation*}
which is used to check the solvability condition for $h_p$. 
On the first step $h_3$ should be find in the form of a third order polynomial in $x$ and $y$ 
which depends depends on $f_2$ and $g_2$.

In the rest the proof is essentially the same as in  Theorem~\ref{interpol}.
\end{proof}

\section{Simplified normal forms: 
$\mu_0=1$ with the non-diagonalisable linear part.}\label{Se:mu1}

In this section we show that the formal Hamiltonian
constructed in the previous section can be substantially
simplified by a formal canonical change of variables
leading to the unique normal form. This results extends
earlier results of \cite{BaiderS1991} onto an one parametric
unfolding of the nilpotent Hamiltonian singularity.

Similar to the previous section we write $h$ as a formal
sum of quasi-homogeneous polynomials.

\begin{proposition}\label{Thm:parabolic_simpl}
If $h(x,y,\varepsilon )= \frac{y^2}{2}+bx^3+\sum_{p\geq 7}h_{p}(x,y,\varepsilon )$
where \[ h_p(x,y,\varepsilon )=\sum_{2k+3l+6m=p}h_{klm}x^ky^l\varepsilon^m \]
then there exists a formal canonical substitution
such that the Hamiltonian takes the form
\begin{equation}\label{Eq:hmu1}
\widetilde{h}(x,y,\varepsilon )=\frac{y^{2}}{2}+ bx^3+\sum_{k+3m\ge 4}u_{km}x^k\varepsilon^m .
\end{equation}
\end{proposition}

\medskip

\begin{proof} We prove the theorem by induction.
The leading order of the Hamiltonian has the form
\begin{equation}\label{Eq:h6}
h_{6}(x,y,\varepsilon )=
\frac{y^{2}}{2}+bx^{3}\,.
\end{equation}

Suppose we transformed the Hamiltonian to the desired form
up to the order $p$ with $p\ge 6$. We will find a quasi-homogeneous
 polynomial $\chi_p$ such that $\widetilde{h}=h \circ \Phi _{\chi_{p }}^{1}$ 
has the desired form up to the order $p+1$.

For any quasi-homogeneous polynomial
$\chi_p$ of order $p$ we have
\begin{eqnarray*}
\widetilde{h} &=&\exp({L_{\chi_p}}h)=
h+\left\{ h,\chi _{p}\right\} +\widehat O_{2p-4} \\
&=&h+\left\{ h_{6},\chi _{p}\right\} +\widehat O_{p+2},
% \\
%&=&h+\frac{\partial u_{6}}{\partial x}\frac{\partial \chi _{p}}{\partial y}%
%-ay\frac{\partial \chi _{p}}{\partial x}+O_{p+2.}
\end{eqnarray*}
where $\widehat O_k$ denotes a formal series which starts
with a quasi-homogeneous polynomial of order $k$.
We see that this change does not affect any term of order less or equal $p$.
Collecting all terms of order $p+1$ we get
\begin{equation}
\tilde{h}_{p+1}=h_{p+1}+\left\{ h_{6},\chi _{p}\right\}
.
\label{Lesson8_4}
\end{equation}%
Let $p+1=3q+r$ with $r\in\{0,1,2\}$. The largest possible power
of $y$ in $h_{p+1}$ does not exceed $q$ and we can write
\begin{equation*}
h_{p+1}=\sum_{\substack {0 \le l\le q\\ (3l+r) \ even}} y^{q -  l} u_{r+3l}(x,\varepsilon)
\qquad\mbox{and}\qquad
\tilde{h}_{p+1}=\sum_{\substack {0 \le l\le q\\ (3l+r) \ even}} y^{q -  l} \tilde{u}_{r+3l}(x,\varepsilon),
\end{equation*}
\begin{equation}\label{chip}
\chi _{p}=\sum_{\substack {0 \le l\le q\\ (3l+r-1) \ even}} y^{q-l}v_{r+3l-1}(x,\varepsilon),
\end{equation}
where $u_j$, $\tilde{u}_j$ and $v_j$ are quasi-homogeneous polynomials
in $x$ and $\varepsilon$ of order $j$. In order to simplify our notation we allow $v_j$ with $j<0$ 
assuming that $v_j=0$ if $j<0$.

Taking into account
\[
\left\{ h_{6},\chi _{p}\right\}=
3bx^2 \frac{\partial
\chi _{p}}{\partial y}-y\frac{\partial \chi _{p}}{\partial x}
\]
and collecting terms of the same order in $y$ we get (excluding the $y$-independent terms) from (\ref{Lesson8_4}):
\begin{equation} \label{url}
\tilde{u}_{r+3l} =u_{r+3l}+3b(q+1-l)x^2v_{r+3l-4}-\frac{\partial v_{r+3l+2}}{\partial x},
\quad 0\le l \le q-1.
\end{equation}
We choose $v_{r+3l+2}$ such that $\tilde{u}_{r+3l}=0$.

If $p+1=3q+r$ is odd then the last equation is at $y^1$ and  $v_{r+3q-1}$ 
is defined such that $\tilde{u}_{r+3q-3}=0$ and so, $\tilde{h}_{p+1}=0$.

If $p+1$ is even then the last equality is at $y^0$:
\begin{equation}
\label{u3q} \tilde{u}_{3q+r}=u_{3q+r}+3bx^2v_{3q+r-4}, \end{equation}
where $v_{3q+r-4}=v_{p-3}$ was defined  from (\ref{url}) on the previous step with $l=q-1$.

We have got $\tilde{u}_j=u_j=0$ except $\tilde{u}_{p+1}$.  So, we can construct $\chi _{p}$ 
such that $\tilde h_{p+1}(x,y,\varepsilon )=\tilde{u}_{p+1}(x,\varepsilon ) $ is independent from $y$.
\end{proof}

\medskip

Theorem~\ref{th1} follows from Theorem~\ref{interpol}, equation (\ref{h6c}) and Proposition~\ref{Thm:parabolic_simpl}.

\medskip

Note that further simplification of the Hamiltonian $h_\varepsilon$ is possible.
Indeed, if in (\ref{url}) $r+3l+2=0 \mod 6$ then $v_{r+3l+2}$ is defined up to 
${\rm const} \, \varepsilon^{(r+3l+2)/6}$. Therefore in (\ref{u3q})
the polynomial $v_{p-3}$ is defined by (\ref{url}) up to an arbitrary 
quasi-homogeneous polynomial in $(x^3, \varepsilon )$. 
So, if $b\ne 0$ we can use this freedom to eliminate $x^k$ for $k=2 \mod 3$ in $\tilde h$.

We will prove a more general result:

\begin{proposition}\label{Thm:parab_simpl4}
Let \[h(x,y,\varepsilon )=\frac{y^2}{2}+bx^n+\sum_{k+nm>n}c_{km}x^k\varepsilon^m \]
with $b\neq 0$, $n\ge 3$.
Then there exists a formal canonical change of variables
such that the Hamiltonian takes the form
\begin{equation}\label{Eq:hmu1n4n}
\widetilde{h}(x,y,\varepsilon )=\frac{y^{2}}{2}+ bx^n+\sum_{\substack {k+nm\ge n\\ k\ne-1\mod n}}u_{km}x^k\varepsilon^m .
\end{equation}
The coefficients of the series are defined uniquely.
If $h$ is even in $(x,y)$ then $\widetilde{h}$ is also even in $(x,y)$.
\end{proposition}

\medskip

\begin{proof}
We will use the following ordering of variables:
\begin{itemize}
\item $x$ is of order 2;
\item $y$ is of order n;
\item $\varepsilon$ is of order 2n.
\end{itemize}
Then $h$ can be written as $ h=\sum_{p \ge 2n} h_p$, where
\begin{eqnarray*} 
h_{2n}&=&\frac{y^{2}}{2}+ bx^n\,,\\
h_p&=&\sum_{2k+2nm=p} c_{km}x^k\varepsilon^m, \quad p>2n\,.
\end{eqnarray*}
Let $\chi_p$ be a quasi-homogeneous polynomial of order $p \ge n+4$. 
After the change of variables given by $\Phi_{\chi_p}^1$ 
the Hamiltonian takes the form
\begin{equation*}
\widetilde{h}=\exp (L_{\chi_p})h=h+\{ h,\chi_p \} + \widehat O_{2p-4}=h+ \{ h_{2n}, \chi_p \} +\widehat O_{p+n-1} +\widehat O_{2p-4}\,.
\end{equation*}
Consequently
\( \widetilde{h}_k =h_k\) for \( k<p+n-2 \)
and
\begin{equation}\label{htilde}
\widetilde{h}_{p+n-2} ={h}_{p+n-2}+ \{ h_{2n}, \chi_p \}.
\end{equation}
We will consider the homological operator
\begin{equation}\label{homol}
L(\chi_p) = \{ h_{2n}, \chi_p \}=nbx^{n-1}\frac{\partial \chi_p}{\partial y}-y \frac{\partial \chi_p}{\partial x}.
\end{equation}
Let $p+n-2=qn +r$ with $r \in \{ 0, 1, \dots , n-1 \}$. Then $\widetilde{h}_{p+n-2} $ and ${h}_{p+n-2}$ can be written as
\[
\widetilde{h}_{p+n-2} = \sum_{0 \le l \le q} y^l \widetilde{u}_{n(q-l)+r} (x, \varepsilon ), \quad
{h}_{p+n-2}= \sum_{0 \le l \le q} y^l {u}_{n(q-l)+r} (x, \varepsilon ),
\]
where $\widetilde{u}_j$ and $u_j$ are quasi-homogeneous polynomials in $(x, \varepsilon )$ of order $j$.
The highest power of $y$ in $\chi_p$ does not exceed
\[
\left[ \frac{p}{n} \right] = \left\{
\begin{array}{ll}
q-1 \quad &\mbox{if }r \in \{ 0, 1, \dots , n-3 \}, \\
q \quad &\mbox{if }r \in \{ n-2,n-1 \}.%
\end{array}%
\right.
\]
Moreover if $r=n-1$ then $p=qn+1$ and $\chi_p$ does not contain the term proportional to $y^q$ 
because the orders of $x$ and $\varepsilon$ are larger or equal 2.
So,
\begin{equation} \label{yui}
\chi_p=\delta_{r,n-2}y^qv_0+\sum_{0 \le l \le q-1} y^l v_{p-nl} (x, \varepsilon ) ,
\end{equation}
where $\delta_{r,n-2}$ is Kronecker symbol.
Calculating $\frac{\partial \chi_p}{\partial y}$ and $\frac{\partial \chi_p}{\partial x} $ 
and substituting them into (\ref{homol}) we get
\begin{eqnarray*}
L(\chi_p) &=&\sum_{0 \le l \le q-2} nb(l+1) y^l x^{n-1}v_{p-n(l+1)}(x, \varepsilon ) 
\\
&&\quad
 +\delta_{r,n-2} nbqy^{q-1}v_0 x^{n-1}  -
\sum_{1 \le l \le q} y^l \frac{\partial v_{p-n(l-1)}}{\partial x}(x, \varepsilon ).
\end{eqnarray*}
Substituting $L(\chi_p)$ into (\ref{htilde}) and considering coefficients of $y^l$ we get
\begin{eqnarray} \label{1}
l=q:&\quad & \widetilde u_r=u_r - \frac{\partial v_{r+2}}{\partial x},
\\
 \label{2}
l=q-1:&&  \widetilde u_{n+r} =u_{n+r}  +\delta_{r,n-2}nbqv_0x^{n-1}
-\frac{\partial v_{n+r+2}}{\partial x},
\\
 \label{3}
1 \le l \le q-2:&& 
\widetilde u_{n(q-l)+r}=u_{n(q-l)+r}
\\&&\qquad\nonumber
+nb(l+1)x^{n-1}v_{n(q-l-2)+r+2}
 -
\frac{\partial v_{n(q-l)+r+2}}{\partial x},
\\
\label{4}
l=0:&& 
\widetilde u_{p+n-2}= u_{p+n-2}+nbx^{n-1}v_{p-n}.
\end{eqnarray}
We can choose $v_{r+2}$ in (\ref{1}) such that $\widetilde u_r=0$. 
Similarly we can choose $v_{n+r+2}$ in (\ref{2}) such that $\widetilde u_{n+r}=0$. 
If $r \ne n-2$ then $v_{n+r+2}$ is unique.

The equations (\ref{1})--(\ref{3}) with $\widetilde u_{j}=0$ can
be considered as two recurrent systems of equations on
$v_{r+2+2kn}$ and on $v_{r+2+n+2kn}$ respectively. If $r \ne n-2$ the polynomial
$v_{p-n}$ is defined uniquely from (\ref{3}) with $l=2$ and, consequently, 
the polynomial $\widetilde u_{p+n-2}$ in (\ref{4}) is also defined uniquely.

If $r=n-2$ then in (\ref{2}) we can take an arbitrary $v_0$ and $v_{n+r+2}=v_{2n}$ contains ${\rm const} \, \varepsilon $. 
So, the polynomial $v_{p-n}=v_{q(n-1)}$  is defined up to an arbitrary
quasi-homogeneous polynomial of order $p-n=n(q-1) $
in $(x^{n}, \varepsilon )$. Consequently, by choosing this polynomial
in (\ref{4}) all terms containing $x^{kn+(n-1)}$ in $\widetilde u_{p+n-2}$ can be eliminated.

We note that if in (\ref{1})--(\ref{4}) $u_j$ are even functions of $x$ then $v_{j+2}$ are odd.
And if $n$ is even then $x^{n-1}$ is odd and $x^{n-1}v_{j+2}$ are even. So, $\widetilde u_{p+n-2}$
are even and $\widetilde h$ is even too.

In order to complete the proof we need to establish uniqueness of
the series (\ref{Eq:hmu1n4n}). Let $\mathcal H_p$ denote the set
of all quasi-homogeneous polynomials of order $p$. Then $L: \mathcal H_p \to \mathcal H_{p+n-2}$.
We have chosen $\chi_p$ on each step such that $\tilde{h}_{p+n-2}$
is in the complement to the range of $L $.

Let us consider the kernel of $L$. Substituting (\ref{yui}) into equation
\[
L(\chi_p)=0
\]
and considering the coefficients in front of $y^{l}$ we get equations (\ref{1})--(\ref{4}) with $\widetilde u_j=u_j=0$.
It's not difficult to see that $\dim \ker L =k$ if $p=2nk$ otherwise $\ker L$ is trivial.
Since $L(h_{2n}^j)=0$ the $\ker L$ consists of $\varepsilon^{k-j}h_{2n}^j$ only.

Suppose there exist two Hamiltonians $\tilde{h}$ and $\tilde{h}'$ of the form
(\ref{Eq:hmu1n4n}). Then there exists $\chi=\sum_{k\ge p}\chi_k$ such that $\tilde{h}'=\exp (L_{\chi}) \tilde{h}$.
 Then $\tilde{h}'_k=\tilde{h}_k$ for $k< p+n-2$ and
\[
\tilde{h}'_{p+n-2}=\tilde{h}_{p+n-2}+L(\chi_p).
\]
Since both $\tilde{h}'_{p+n-2}$ and $\tilde{h}_{p+n-2}$ are in the complement
of the range of $L$, $\tilde{h}'_{p+n-2}=\tilde{h}_{p+n-2}$ and $L(\chi_p)=0$.
 Consequently, $\chi_p=0$ if $p\ne 0\mod 2n$ or $\chi_p=\sum_{1\le j \le k}\alpha_j \varepsilon^{k-j}h_{2n}^j$ if $p=2nk$.
 It means that
\[\chi=\sum_{1\le j \le k}\alpha_j \varepsilon^{k-j}h_{2n}^j + \sum_{j\ge p+1} \chi_j\]
and the $\tilde{h}'$ and $\tilde{h}$ are also connected by $\tilde{\chi}=\sum_{j\ge p+1} \chi_j$.
The order of $\tilde \chi$ is at least $p+1$ and we can repeat the arguments to show
that $\tilde{h}'_{p+2}=\tilde{h}_{p+2}$ and so on. We conclude by induction
that $\tilde{h}'$ coincides with $\tilde{h}$ at all orders.
\end{proof}

Theorem~\ref{th3} follows from Theorem~\ref{th1} and Proposition~\ref{Thm:parab_simpl4}.

\section{Birkhoff normal form for $\mu_0=-1$}\label{Se:Birkhoff}

Now we consider the case $\mu_0 =-1$.
As in the case of $\mu_0=1$ the Jacobian of $F_0$ has
a double eigenvalue and can be non-diagonalisable.
The following theorem follows from the classical results
and is included for the sake of completeness.

\begin{theorem}\label{Thm:Birkhoff}
If $F_{\varepsilon }(x,y)$ is a family of area-preserving maps such that
$F_{0}(0)=0$ and
$
DF_{0}(0)=\left(
\begin{array}{cc}
-1 & -1 \\
0 & -1%
\end{array}
\right)$,
then there exists a canonical formal coordinate change
$\phi $ such that
\begin{equation*}
\widetilde F_\varepsilon =\phi ^{-1}\circ F_{\varepsilon }\circ \phi
\end{equation*}
is an odd series in the space variables.
\end{theorem}

\begin{proof} We will use the following ordering:
\begin{equation*}
\begin{array}{cc}
x & \text{is of order }2, \\
y & \text{is of order }3, \\
\varepsilon  & \text{is of order }6.
\end{array}
\end{equation*}
The proof goes by induction, we normalise the orders $\left( p-3,p-2\right) $
consecutively using quasi-homogeneous Hamiltonian flows
\begin{equation*}
\chi _{p}(x,y,\varepsilon )=\sum_{2k+3l+6m=p}c_{klm}x^{k}y^{l}\varepsilon ^{m}.
\end{equation*}
The corresponding Lie series are well defined for $p\geq 6$.
Now consider the change of variables
\begin{equation}
\widetilde{F}_{\varepsilon }=\Phi _{\chi _{p}}^{-1}\circ F_{\varepsilon
}\circ \Phi _{\chi _{p}}^{1},  \label{Lesson8_1}
\end{equation}
we will use the polynomial $\chi _{p}$ to eliminate even terms at the order
$\left( p-3,p-2\right) $.
Equation (\ref{Lesson8_1}) is equivalent to
\begin{equation}\label{Eq:FPhi}
\Phi _{\chi _{p}}^{1}\circ \widetilde{F}_{\varepsilon }=F_{\varepsilon
}\circ \Phi _{\chi _{p}}^{1}\,.
\end{equation}
We have
\begin{eqnarray*}
F_{\varepsilon }\circ \Phi _{\chi _{p}}^{1} &=&F_{\varepsilon }+\left\{
F_{\varepsilon },\chi _{p}\right\} +\widehat{O}_{(2p-8,2p-7)} \\
&=&F_{\varepsilon }+\left(
\begin{array}{c}
-\frac{\partial \chi _{p}}{\partial y} \\
\frac{\partial \chi _{p}}{\partial x}
\end{array}
\right) +\widehat{O}_{(p-2,p-1)},
\end{eqnarray*}
where $\widehat{O}_j$ denotes a series which starts
with a quasi-homogeneous order at least $j$ or higher.
Taking into account the form of the map we get
\begin{equation}\label{Eq:lhs}
F_{\varepsilon }\circ \Phi _{\chi _{p}}^{1}\left( x,y\right) =\left(
\begin{array}{c}
-x-y+f_{\varepsilon }(x,y)-\frac{\partial \chi _{p}}{\partial y}(x,y)+%
\widehat{O}_{p-2} \\
-y+g_{\varepsilon }(x,y)+\frac{\partial \chi _{p}}{\partial x}(x,y)+\widehat{%
O}_{p-1}%
\end{array}%
\right) .
\end{equation}%
The time-one map is given by%
\begin{equation*}
\Phi _{\chi _{p}}^{1}\left(
x ,
y%
\right) =\left(
\begin{array}{c}
x+\frac{\partial \chi _{p}}{\partial y}+\widehat{O}_{2p-8} \\
y-\frac{\partial \chi _{p}}{\partial x}+\widehat{O}_{2p-7}%
\end{array}%
\right) .
\end{equation*}%
Suppose
\begin{equation*}
\widetilde{F}_{\varepsilon }\left(
x ,
y%
\right) =\left(
\begin{array}{c}
-x-y+\widetilde{f}_{\varepsilon }(x,y) \\
-y+\widetilde{g}_{\varepsilon }(x,y)%
\end{array}%
\right) .
\end{equation*}%
Then we have
\begin{equation*}
\Phi _{\chi _{p}}^{1}\circ \widetilde{F}_{\varepsilon }\left( x,y\right) =\left(
\begin{array}{c}
-x-y+\widetilde{f}_{\varepsilon }(x,y)+\left. \frac{\partial \chi _{p}}{%
\partial y}\right\vert _{(-x,-y)}+\widehat{O}_{p-2} \\
-y+\widetilde{g}_{\varepsilon }(x,y)-\left. \frac{\partial \chi _{p}}{%
\partial x}\right\vert _{(-x,-y)}+\widehat{O}_{p-1}%
\end{array}%
\right) .
\end{equation*}%
Substituting this and (\ref{Eq:lhs}) in (\ref{Eq:FPhi}) and collecting the terms of
 the order $(p-3,p-2) $ we get equations for $\chi_p$:
\begin{equation*}
\left\{
\begin{array}{c}
\left[ \widetilde{f}_{\varepsilon }(x,y)\right] _{p-3}=\left[ f_{\varepsilon
}(x,y)\right] _{p-3}-\frac{\partial \chi _{p}}{\partial y}(x,y)-\frac{%
\partial \chi _{p}}{\partial y}(-x,-y), \\[12pt]
\left[ \widetilde{g}_{\varepsilon }(x,y)\right] _{p-2}=\left[ g_{\varepsilon
}(x,y)\right] _{p-2}+\frac{\partial \chi _{p}}{\partial x}(x,y)+\frac{%
\partial \chi _{p}}{\partial x}(-x,-y).%
\end{array}%
\right.
\end{equation*}%
We see in the equation that all even terms of $\partial_y \chi_p$
and $\partial_x \chi_p$  are doubled and all odd terms are
cancelled. Let us now choose $\chi _{p}$ such that all even terms
of $\widetilde{f}_{\varepsilon }$ and $\widetilde{g}_{\varepsilon }$ are
eliminated. Let $\chi _{p}$ be an odd polynomial such that%
\begin{equation}
\left\{
\begin{array}{c}
\frac{\partial \chi_{p} }{\partial y}=\frac{1}{2}\left[ f_{\varepsilon }(x,y)%
\right] _{p-3}^{even}, \\[12pt]
-\frac{\partial \chi _{p}}{\partial x}=\frac{1}{2}\left[ g_{\varepsilon
}(x,y)\right] _{p-2}^{even},%
\end{array}%
\right.   \label{Lesson8_2}
\end{equation}%
then $\left[ \widetilde{f}_{\varepsilon }(x,y)\right] _{p-3}$ 
and $\left[ \widetilde{g}_{\varepsilon }(x,y)\right] _{p-2}$
are odd. A sufficient condition for existence of $\chi _{p}$ is zero
divergence which follows from the area-preservation property of 
$F_{\varepsilon }$.

Suppose all orders of $F_{\varepsilon }$ up to $\left( p-4,p-3\right) $ are
odd. Since $F_{\varepsilon }$ is area-preserving we have%
\begin{equation*}
\det DF_{\varepsilon } =\left\vert
\begin{array}{cc}
-1 +\partial _{x}f_{\varepsilon } & -1+\partial _{y}f_{\varepsilon } \\
\partial _{x}g_{\varepsilon } & -1 +\partial _{y}g_{\varepsilon }%
\end{array}%
\right\vert
=1
\end{equation*}
which is equivalent to
\begin{equation}
- \left( \partial _{x}f_{\varepsilon }+\partial _{y}g_{\varepsilon }\right) 
+\left\{ f_{\varepsilon },g_{\varepsilon }\right\}
+ \partial _{x}g_{\varepsilon }\equiv0.
 \label{Parabolic_ap}
\end{equation}
Taking
order $p-5$ we get%
\begin{equation*}
-\frac{\partial }{\partial x}f_{p-3}-\frac{\partial }{\partial y}%
g_{p-2}+\sum_{\substack{ k+l-5=p-5 \\ k,l\geq 4}}\left\{ f_{k},g_{l}\right\}
+\frac{\partial }{\partial x}g_{p-3}=0.
\end{equation*}%
Note that the last two terms are even, hence we
conclude that%
\begin{equation*}
\frac{\partial f_{p-3}^{even}}{\partial x}+\frac{\partial g_{p-2}^{even}}{%
\partial y}=0.
\end{equation*}%
Hence $\left( \ref{Lesson8_2}\right) $ has a solution $\chi _{p}.$

\medskip

We have proved that the even part of the order $(p-3,p-2)$
with $p\ge6$ can be eliminated provided
all previous orders are odd.
To complete the proof we need a base for the induction.
We note that the expansion of the map starts with the order
$(2,3)$.  There are no even terms of order $3$ or lower (the only
term of order $2$ is $x$, and the only term of order $3$ is $y$). Thereby
the induction base is established.
\end{proof}

\section{Unique normal form for $\mu_0 =-1$}\label{Se:mu-1}
Now we consider $\mu_0 =-1$. We have
\begin{equation*}
F_{\varepsilon }(x,y)=
\left(
\begin{array}{c}
-x- y+f(x,y,\varepsilon ) \\
-y+g(x,y,\varepsilon )%
\end{array}%
\right)  =-\Phi _{h}^{1}(x,y).
\end{equation*}%
In the case $\mu_0=-1$ the family $F_{\varepsilon }$ can be
transformed to the Birkhoff normal form. Then $F_{\varepsilon }$ is odd (Theorem~\ref{Thm:Birkhoff})
and $h(x,y,\varepsilon )=\sum_{p\geq 4}h_{p}(x,y,\varepsilon )$ 
is an even function on the space variables (Theorem~\ref{interpol}).
Then due to Proposition~\ref{Thm:parab_simpl4} there exists a formal canonical substitution
such that the Hamiltonian takes the form
\begin{equation*}\label{Eq:heven}
\widetilde{h}(x,y,\varepsilon )=\frac{y^{2}}{2}+\sum_{k+m\geq
2}u_{km}x^{2k}\varepsilon^m
\end{equation*}
and the coefficients of the series are defined uniquely.
Theorem~\ref{th4} is proved.

\section{Diagonalisable Jacobian}\label{Se:diag}
We will use the following ordering:
\begin{equation*}
\begin{array}{cc}
x & \text{is of order }1, \\
y & \text{is of order }1, \\
\varepsilon  & \text{is of order }3.%
\end{array}%
\end{equation*}
By Theorem~\ref{interpoldiag} there exists a Hamiltonian $h_{\varepsilon}(x,y)=h_3+\sum_{p\ge 4} h_p$ such that $F=\Phi_h^1$.
The polynomial of the third order $h_3$ depends only on $(x,y)$.
After a linear change it can be written as $h_3=\alpha x^3+\beta xy^2$
(it is convenient to look for this change in two steps: first make
a substitution $(x,y)\mapsto (x+cy,y)$ and then $(x,y)\mapsto (x,cx+y)$
swapping $x$ and $y$ using $(x,y)\mapsto(y,-x)$ if necessary).

\begin{proposition}\label{Thm:parab_simpl diag}
If a formal Hamiltonian has the form 
\[h(x,y,\varepsilon )=a x^3+ xy^2+\sum_{k+l+3m \ge 4}h_{klm}x^ky^l\varepsilon^m \,,\]
then there is a formal tangent-to-identity canonical change of variables
such that it is transformed into
\begin{equation}\label{Eq:hdiag}
\widetilde{h}(x,y,\varepsilon )=ax^3+xy^2
+A(x,\varepsilon) + y B(xy^2,\varepsilon)
\end{equation}
where
\begin{equation}
A(x,\varepsilon)=\sum_{k+3m \ge 4}a_{km}x^k\varepsilon^m\,,
\qquad 
B(xy^2,\varepsilon)= \sum_{k+m \ge 1} b_{km} {(xy^2)}^k \varepsilon^m \,.
\end{equation}
The coefficients of the series are defined uniquely.
If $h$ is odd in $x$ then $\widetilde{h}$ is also odd in $x$.
\end{proposition}

\medskip

\begin{proof}
We will use $\chi_p$ for simplifying $h_{p+1}$.
Collecting all terms of order $p+1$ we get
\begin{equation}\label{homodiag}
\widetilde{h}_{p+1}=h_{p+1}+\left\{ h_{3},\chi _{p}\right\}
.
\end{equation}
Let
\begin{equation*}
\chi _{p}=\sum_{l=0}^p y^{p-l}v_{l}(x,\varepsilon),
\end{equation*}
\[
h _{p+1}=\sum_{l=0}^{p+1} y^{p+1-l}u_{l}(x,\varepsilon),
\quad \widetilde h _{p+1}=\sum_{l=0}^{p+1} y^{p+1-l}\widetilde u_{l}(x,\varepsilon),
\]
where $v_l$, $u_l$ and $\widetilde u_l$ are quasi-homogeneous polynomials
in $x$ and $\varepsilon$ of order $l$.
Substituting $\widetilde{h}_{p+1}$, $h_{p+1}$ and $\chi _{p} $ into (\ref{homodiag}) 
we get $p+2$ equalities (for $l=0,1,...,p,p+1$):
\begin{eqnarray} \label{udiag1} \widetilde u_0 &=&u_0 +pv_0, \\
\label{udiag2} \widetilde u_1 &=&u_1 + \left[ (p-1) v_1 -2xv'_1 \right], \\
\label{udiag3} \widetilde u_l &=&u_l +3a(p-l+2)x^2 v_{l-2} + \left[ (p-l)v_l -2xv_l' \right] ,
\quad 2 \le l \le p-1, \\
 \label{udiag4} \widetilde u_p &=&u_p +6ax^2v_{p-2} -2xv_p' ,\\
\label{udiag5} \widetilde u_{p+1} &=&u_{p+1} +3ax^2 v_{p-1} .
\end{eqnarray}
Set $v_0=-\frac{u_0}{p}$ to make $\widetilde u_0 =0$. 
Let 
\begin{eqnarray*}
v_l(x,\varepsilon )&=&\sum_{m=0}^{\left[ l/3\right]} \nu_m x^{l-3m} \varepsilon ^m,\\
\\
u_l(x,\varepsilon )&=&\sum_{m=0}^{\left[ l/3\right]} \upsilon_m x^{l-3m} \varepsilon ^m
\,,\\
\widetilde u_l(x,\varepsilon )&=&\sum_{m=0}^{\left[ l/3\right]} \widetilde \upsilon_m x^{l-3m} \varepsilon ^m.
\end{eqnarray*}
We will choose $\nu_m$ such that $\widetilde u_l$ takes the simplest form, 
i.e. $\widetilde \upsilon_m$ vanishes if possible.
Equations (\ref{udiag2})--(\ref{udiag3}) for $v_l$ can be written as
\[ 2xv_l'-(p-l)v_l=f_l(x, \varepsilon ), \qquad 1\le l \le p-1,
\]
where $f_l (x, \varepsilon ) =u_l-3a(p-l+2)x^2v_{l-2}$, $v_{-1}=0$.

Let $f_l(x,\varepsilon )=\sum_{m=0}^{\left[ l/3\right]} \varphi_m x^{l-3m} \varepsilon ^m$.
Then for $0 \le m \le \left[ l/3-1/3 \right] $
\[ \nu_m \left[ 3l-6m-p \right] =\varphi_m
\]
If $l$, $m$, $p$ are such that $3l-6m-p \ne 0$ then $\nu_m$ can be chosen such that 
$\widetilde \upsilon_m =0$. But if $3l-6m-p=0$ 
then we cannot change $\widetilde \upsilon_m$ by choosing $\nu_m$.
It takes place if $p=3s$, $l=s+2m$, i.e. $h$ and $\widetilde h$ have the same coefficients 
in front of $y^{2k+1}x^k\varepsilon^m$ (for $k\ge 1$, $m\ge 0$).

In (\ref{udiag4}) we put $v_p(x,\varepsilon )=\sum_{m=0}^{\left[ p/3\right]} \nu_m x^{p-3m} \varepsilon ^m$. Then
\[x v_p'(x,\varepsilon )=\sum_{m=0}^{\left[ p/3\right]} (p-3m) \nu_m x^{p-3m} \varepsilon ^m . \]
If $p-3m \ne 0$ we can chose $\nu_m$ such that the corresponding terms in $\widetilde u_p$ vanish.

If $p-3m=0$ then in $\widetilde u_p$ does not vanish the term with $x^0 \varepsilon^{p/3}$, 
i.e. in $\widetilde h_{p+1}$ does not vanish the term with  $y\varepsilon^{p/3}$.
In (\ref{udiag5}) $v_{p-1}$ is already defined from (\ref{udiag4}) with $l=p-1$. 
Therefore $\widetilde u_{p+1}$ does not vanish in general case.

So, $\widetilde h_{p+1}$ contains terms with $x^k\varepsilon^m$ and $y{(xy^2)}^k\varepsilon^m$.

The uniqueness follows from the usual analysis of the complement to the range of
the homological operator based on the equations (\ref{udiag1})--(\ref{udiag5}) 
with $u_l=\widetilde u_l=0$.
\end{proof}

\section{Orientation reversing families} \label{orientrev}

Now we will consider the area-preserving  orientation-reversing families. 

\begin{theorem}\label{Thm:Birkhoffdiag}
If $F_{\varepsilon }(x,y)$ is a family of area-preserving maps such that
$F_{0}(0)=0$ and
$
DF_{0}(0)=\left(
\begin{array}{cc}
-1 & 0 \\
0 & 1%
\end{array}
\right)$,
then there exists a canonical formal coordinate change
$\phi $ such that
\begin{equation*}
\widetilde F_\varepsilon =\phi ^{-1}\circ F_{\varepsilon }\circ \phi= \left(
\begin{array}{c}
f(x,y,\varepsilon ) \\
 g(x,y,\varepsilon )%
\end{array}%
\right) ,
\end{equation*}%
where $f$ is odd and $g$ is even in $x$.
\end{theorem}

\begin{proof}   We will use the following ordering:
\begin{equation*}
\begin{array}{cc}
x & \text{is of order }1, \\
y & \text{is of order }1, \\
\varepsilon  & \text{is of order }3.%
\end{array}%
\end{equation*}
In the rest of the proof is completely analogous to the proof of Theorem~\ref{Thm:Birkhoff}.
\end{proof}

We can write the map in the form
\begin{equation*}
F_{\varepsilon }\left(
\begin{array}{c}
x \\
y%
\end{array}%
\right) = \left(
\begin{array}{c}
-x+ \sum_{p\ge 2}f_p(x,y,\varepsilon ) \\
 y+\sum_{p \ge 2}g_p(x,y,\varepsilon )%
\end{array}%
\right) ,
\end{equation*}%
where $f_p$ is odd and $g_p$ is even in $x$, therefore $h$ is odd in $x$ 
(a proof is similar to Theorem~\ref{interpoldiag}), then $h_3=ax^3+bxy^2$. 
By Proposition~\ref{Thm:parab_simpl diag} there exists a formal canonical change of space variables
such that the Hamiltonian takes the form
\[
\widetilde h (x,y,\varepsilon )=ax^3+xy^2 + \sum_{k+3m \ge 4} 
a_{mk}\varepsilon^m x^{2k-1} +y \sum_{k+m \ge 1} b_{mk} \varepsilon^m {(xy^2)}^{2k-1} 
\]
and Theorem~\ref{th5} follows.

\end{document}